\documentclass[mathpazo]{cicp}

\usepackage{lipsum}
\usepackage{amsfonts}
\usepackage{graphicx}
\usepackage{epstopdf}
\usepackage{algorithm}
\usepackage{algorithmic}
\usepackage{tikz}
\usetikzlibrary{arrows}

\usepackage{multirow}
\usepackage{listings}
\usepackage{mathtools}
\usepackage{latexsym,amsmath,amsfonts,amscd}
\usepackage{subfigure}
\usepackage{bbm}

\newtheorem{Proposition}{Proposition}
\newtheorem{rmk}{Remark}
\newtheorem{defi}{Definition}

\begin{document}
\title{On the  monotonicity of $Q^2$ spectral element method for Laplacian on quasi-uniform rectangular meshes}


\author{Logan J. Cross\affil{1}  and Xiangxiong Zhang\affil{1}\comma\corrauth,
}
\address{\affilnum{1}\  
         Purdue University,
150 N. University Street,
West Lafayette, IN 47907-2067.}
\emails{{\tt logancross68@gmail.com} (Cross),  
         {\tt zhan1966@purdue.edu} (Zhang)}

\begin{abstract}The monotonicity of discrete Laplacian implies discrete maximum principle, which in general does not hold for high order schemes.  
The $Q^2$ spectral element method  has been proven monotone on a uniform rectangular mesh.  
In this paper we  prove the monotonicity of the $Q^2$ spectral element method  on quasi-uniform rectangular meshes under certain mesh constraints. In particular, we propose a relaxed Lorenz's condition for proving monotonicity.

\end{abstract}


\maketitle

\section{Introduction}

In many applications, monotone discrete Laplacian operators are desired and useful for ensuring stability such as discrete maximum principle or positivity-preserving of physically positive quantities \cite{ciarlet1970discrete, shen2021discrete, hu2021positivity, liu2023positivity}. Let $\Delta_h$ denote the matrix representation of a discrete Laplacian operator, then it is called {\it monotone} if $(-\Delta_h)^{-1}\geq 0$, i.e., the inverse matrix $(-\Delta_h)^{-1}$ has nonnegative entries. In this paper, all inequalities for matrices are entry-wise inequalities. 

In the literature, the most important tool for proving monotonicity is via nonsingular M-matrices, which are inverse-positive matrices. See the Appendix for a convenient characterization of the M-matrices.
The simplest second order accurate centered finite difference $u''(x_i)\approx \frac{u(x_{i-1})-2u(x_i)+u(x_{i+1})}{\Delta x^2}$ is monotone because the corresponding matrix  $(-\Delta_h)^{-1}$ is an M-matrix thus inverse positive.
\textcolor{black}{Even though the linear finite element
method forms an M-matrix on unstructured triangular meshes under a mild mesh constraint \cite{xu1999monotone}, 
in general the discrete maximum principle is not true for high order finite element methods  on unstructured meshes \cite{hohn1981some}.} On the other hand, there exist a few high order accurate inverse positive schemes on structured meshes.

For solving a Poisson equation, \textcolor{black}{provably monotone} high order accurate schemes on structured meshes include  the classical 9-point scheme \cite{krylov1958approximate, collatz1960, bramble1962formulation} in which the stiffness matrix is an M-matrix. The classical 9-point scheme has the same stiffness matrix as fourth order accurate
compact finite difference schemes \cite{li2018high}, see the appendix in \cite{li2023-compact}.
In \cite{bramble1964finite,bramble1964new},  a fourth order accurate finite difference scheme
was constructed and its stiffness matrix is a product of two M-matrices thus monotone. 
The Lagrangian $P^2$  finite element method  on a regular triangular mesh \cite{whiteman1975lagrangian} has a monotone stiffness matrix \cite{lorenz1977inversmonotonie}. On an equilateral triangular mesh, the discrete maximum principle of $P^2$ element can also be proven \cite{hohn1981some}. 
 Monotonicity was also proven  for the $Q^2$ spectral element method on an uniform rectangular mesh  for a variable coefficient Poisson equation under suitable mesh constraints \cite{li2019monotonicity}. The $Q^k$ spectral element method is the continuous finite element method with Lagrangian $Q^k$ basis implemented by $(k+1)$-point Gauss-Lobatto quadrature. 
 The monotonicity of $Q^3$ spectral element method for Laplacian on uniform meshes was also proven in \cite{cross2020monotonicity}.

For proving inverse positivity, the main viable tool in the literature is to use M-matrices which are inverse positive.
A convenient sufficient condition for verifying the M-matrix structure is to require that off-diagonal entries must be non-positive. Except the fourth order compact finite difference,
all high order accurate schemes induce positive off-diagonal entries, destroying M-matrix structure, which is a major challenge of proving monotonicity. In  \cite{bramble1964new} and \cite{bohl1979inverse}, and also the appendix  in \cite{li2019monotonicity}, M-matrix factorizations of the form $(-\Delta_h)^{-1}=M_1M_2$ were shown for special high order schemes but these M-matrix factorizations seem ad hoc and do not apply to other schemes or other equations. 
  In \cite{lorenz1977inversmonotonie},
Lorenz proposed some matrix entry-wise inequality for ensuring a matrix to be a product of two M-matrices and applied it to $P^2$ finite element method on uniform regular triangular meshes.

In  \cite{li2019monotonicity}, Lorenz's condition was applied to  $Q^2$ spectral element method on uniform rectangular meshes. Such a monotonicity result implies that the $Q^2$ spectral element method is bound-preserving or positivity-preserving for convection diffusion equations including 
the Allen-Cahn equation
\cite{shen2021discrete}, the Keller-Segel equation \cite{hu2021positivity}, the Fokker-Planck equation \cite{liu2022structure}, as well as the internal energy equation in compressible Navier-Stokes system \cite{liu2023positivity}. On the other hand, all these results about $Q^2$ spectral element method are on uniform meshes. For both theoretical and practical interests, a natural question to ask is whether such a monotonicity result still holds on non-uniform meshes. The monotonicity of high order schemes on  quasi-uniform meshes \textcolor{black}{are preferred} in many applications, e.g., \cite{sulman2019positivity}.

The focus of this paper is to discuss Lorenz's condition for $Q^2$ spectral element method on quasi-uniform meshes. 
We
discuss and derive sufficient mesh constraints to preserve monotonicity of  $Q^2$ spectral element method on a quasi-uniform rectangular mesh. In general, the same discussion also applies to Lagrangian $P^2$ finite element method on a quasi-uniform regular triangular mesh, 
 but there does not seem to be any advantage of using  $P^2$.  

For simplicity, we will focus only on Dirichlet boundary conditions. For Neumann boundary conditions, the discussion of \textcolor{black}{monotonicity} is very similar, e.g., see \cite{hu2021positivity, liu2022structure} for discussion on Neumann boundaries. 
 
The rest of the paper is organized as follows. In Section \ref{sec:sem}, we briefly review the $Q^2$ spectral element method and its equivalent finite difference form for the Poisson equation.  
In Section \ref{sec-lorenz}, we review the Lorenz's condition for proving monotonicity and propose a relaxed version of Lorenz's condition. \textcolor{black}{Though we only focus on $Q^2$ spectral element method on quasi-uniform meshes for Laplacian in this paper, the proposed relaxed Lorenz's condition may also be used to derive monotonicity under more relaxed mesh constraints for $Q^2$ spectral element method solving variable coefficient problems such as those in \cite{li2019monotonicity, hu2021positivity, liu2022structure}. } 
In Section \ref{sec:main}, we prove the monotonicity of $Q^2$ spectral element method  on a quasi-uniform mesh by using the relaxed Lorenz's condition.  
Numerical tests of accuracy of the scheme and necessity of the  mesh constraints for monotonicity are given in Section \ref{sec:5}. Section \ref{sec:remark} are concluding remarks.

 \section{$Q^2$ spectral element method}
 \label{sec:sem}
 
 \subsection{Finite element method with the simplest quadrature}
 Consider an elliptic equation on $\Omega=(0,1)\times(0,1)$ with 
Dirichlet boundary conditions:
\begin{equation}
\label{pde-3}
 \mathcal L u\equiv -\nabla\cdot(a \nabla u)+cu =f  \quad \mbox{on} \quad \Omega, \quad
 u=g \quad \mbox{on}\quad  \partial\Omega. 
\end{equation}
Assume there is a function $\bar g \in H^1(\Omega)$ as an extension of $g$ so that $\bar g|_{\partial \Omega} = g$.
 The  variational form  of \eqref{pde-3} is to find $\tilde u = u - \bar g \in H_0^1(\Omega)$ satisfying
\begin{equation}\label{nonhom-var}
 \mathcal A(\tilde u, v)=(f,v) - \mathcal A(\bar g,v) ,\quad \forall v\in H_0^1(\Omega),
 \end{equation}
 where $\mathcal A(u,v)=\iint_{\Omega} a \nabla u \cdot \nabla v dx dy+\iint_{\Omega} c u v dx dy$, $ (f,v)=\iint_{\Omega}fv dxdy.$

   \begin{figure}[h]
 \subfigure[The  quadrature points and a finite element mesh]{\includegraphics[scale=0.8]{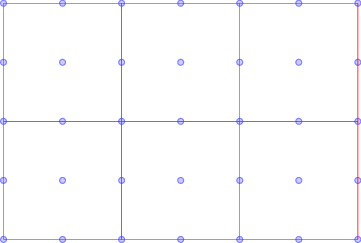} }
 \hspace{.6in}
 \subfigure[The corresponding finite difference grid]{\includegraphics[scale=0.8]{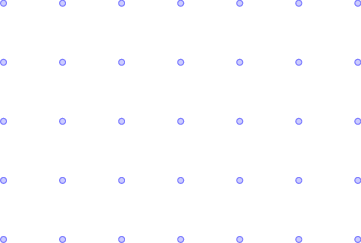}}
\caption{An illustration of Lagrangian $Q^2$ element and the $3\times3$ Gauss-Lobatto quadrature. }
\label{mesh-Q2}
 \end{figure}

 Let $h$ be quadrature point spacing of   a  rectangular mesh shown in Figure \ref{mesh-Q2} and  $V_0^h\subseteq H^1_0(\Omega)$ be the continuous finite element space consisting of $Q^2$ polynomials, then the most convenient implementation of  finite element method is to use the simple quadrature consisting of   $3\times3$ Gauss-Lobatto quadrature rule  for all the integrals,  see    Figure \ref{mesh-Q2} for $Q^2$ method. 
 Such a numerical scheme can be defined as:  find $ u_h \in V_0^h$ satisfying
\begin{equation}\label{nonhom-var-num3}
\mathcal A_h( u_h, v_h)=\langle f,v_h \rangle_h - \mathcal A_h( g_I,v_h) ,\quad \forall v_h\in V_0^h,
\end{equation}
where   $\mathcal A_h(u_h,v_h)$  and $\langle f,v_h\rangle_h$ denote using simple quadrature for integrals $\mathcal A(u_h,v_h)$ and $(f,v_h)$ respectively, and $g_I$ is the piecewise   $Q^2$ Lagrangian interpolation polynomial at the quadrature points shown   Figure \ref{mesh-Q2}  of the following function:
\[g(x,y)=\begin{cases}
   0,& \mbox{if}\quad (x,y)\in (0,1)\times(0,1),\\
   g(x,y),& \mbox{if}\quad (x,y)\in \partial\Omega.\\
  \end{cases}
\] 

Then $\bar u_h =   u_h + g_I$ is the numerical solution for the problem \eqref{pde-3}.
Notice that \eqref{nonhom-var-num3} is not a straightforward approximation to \eqref{nonhom-var} since $\bar g$ is never used. 
When the numerical solution is represented by a linear combination of Lagrangian interpolation polynomials at the grid points, it can be rewritten as a finite difference scheme.
We can also call it a variational difference scheme since it is derived from the
variational form. 
  
\subsection{The difference formulation}
The   scheme \eqref{nonhom-var-num3} with Lagrangian $Q^2$ basis can also be written as  a finite difference scheme \cite{li2019fourth}.

 Consider a uniform grid $(x_i,y_j)$ for a rectangular domain $[0,1]\times[0,1]$
where
$x_i = ih$, $i = 0,1,\dots, n+1$
and
$y_j = jh$, $j = 0,1,\dots, n+1$, $h=\frac{1}{n+1}$, where $n$ must be odd.
Let $u_{ij}$ denote the numerical solution at $(x_i, y_j)$.
Let $\mathbf u$ denote an abstract vector consisting of $u_{ij}$ for $i,j=1,2,\cdots,n$. Let $\bar{\mathbf u}$ denote an abstract vector consisting of $u_{ij}$ for $i,j=0,1,2,\cdots,n,n+1$.
Let $\bar{\mathbf f}$ denote an abstract vector consisting of $f_{ij}$ for $i,j=1,2,\cdots,n$ and the boundary condition $g$ at the boundary grid points. 
Then the matrix vector representation of  \eqref{nonhom-var-num3} is
 $S\bar{\mathbf u}=M\mathbf f$ where $S$ is the stiffness matrix and $M$ is the lumped mass matrix.  
For convenience,  after inverting the mass matrix, with the boundary conditions, the whole scheme can be represented in a matrix vector form 
$\bar L_h \bar{\mathbf u} =\bar{\mathbf f}$.
For Laplacian $\mathcal L u=-\Delta u$,  $\bar L_h \bar{\mathbf u} =\bar{\mathbf f}$ on a uniform mesh is given as 
\begin{equation}
\label{Q2-laplacian}
\resizebox{\textwidth}{!} 
{$ 
    \begin{gathered}
(\bar L_h \bar{\mathbf u})_{i,j}:=\frac{-u_{i-1,j}-u_{i+1,j}+4u_{i,j}-u_{i,j+1}-u_{i+1,j}}{h^2}=f_{i,j},\quad \text{if $(x_i,y_j)$ is a cell center}, \\
(\bar L_h \bar{\mathbf u})_{i,j}:=\frac{-u_{i-1,j}+2u_{i,j}-u_{i+1,j}}{h^2}+\frac{u_{i,j-2}-8u_{i,j-1}+14u_{i,j}-8u_{i,j+1}+u_{i,j+2}}{4h^2}=f_{i,j},
\\ \text{if $(x_i,y_j)$ is an edge center for an edge parallel to the \textcolor{black}{x-axis},}
\\ 
(\bar L_h \bar{\mathbf u})_{i,j}:=\frac{u_{i-2,j}-8u_{i-1,j}+14u_{i,j}-8u_{i+1,j}+u_{i+2,j}}{4h^2}+\frac{-u_{i,j-1}+2u_{i,j}-u_{i,j+1}}{h^2}=f_{i,j},
\\ \text{if $(x_i,y_j)$ is an edge center for an edge parallel to the \textcolor{black}{y-axis},}
\\  
(\bar L_h \bar{\mathbf u})_{i,j}:=\frac{u_{i-2,j}-8u_{i-1,j}+14u_{i,j}-8u_{i+1,j}+u_{i+2,j}}{4h^2}+\frac{u_{i,j-2}-8u_{i,j-1}+14u_{i,j}-8u_{i,j+1}+u_{i,j+2}}{4h^2}=f_{i,j}, \\
 \text{if $(x_i,y_j)$ is a knot,}\\
(\bar L_h \bar{\mathbf u})_{i,j}:=u_{i,j}=g_{i,j}\quad \text{if $(x_i,y_j)$ is a boundary point.}
    \end{gathered}$}
\end{equation}
If ignoring the denominator $h^2$, then the stencil can be represented as:
 \[ \quad \mbox{  cell center}  \begin{array}{ccc}
& -1& \\
-1 & 4 & -1\\
& -1& 
\end{array}\qquad
\mbox{knots}  \begin{array}{ccccc}
&& \frac14& &\\
&& -2& &\\
\frac14& -2 & 7 & -2 &\frac14\\
&& -2& &\\
&& \frac14& &
\end{array}
\]
\[
\mbox{edge center (edge parallel to $y$-axis)}  \begin{array}{ccccc}
&& -1& &\\
\frac14& -2 & \frac{11}{2} & -2 &\frac14\\
&& -1& &
\end{array}\]
\[
\mbox{edge center (edge parallel to $x$-axis)}  \begin{array}{ccccc}
&& \frac14& &\\
&& -2& &\\
& -1 & \frac{11}{2} & -1 &\\
&& -2& &\\
&& \frac14& &
\end{array}.\]
 \begin{remark}
When regarded as a finite difference scheme, the   scheme \eqref{nonhom-var-num3}   is  fourth order accurate  in $\ell^2$-norm for elliptic, parabolic, wave and Schr\"odinger equations \cite{li2019fourth, MR4378595}. 
 \end{remark}

\section{Lorenz's condition for monotonicity} 
\label{sec-lorenz}
 
In this section, we first review the Lorenz's method for proving monotonicity \cite{lorenz1977inversmonotonie}, then present a relaxed Lorenz's condition. The definition of M-matrices is given in the appendix. 

\subsection{Discrete maximum principle} 
\label{sec-dmp}
We first review how the monotonicity implies the discrete maximum principle for a boundary value problem. 
For a finite difference scheme, assume there are $N$ grid points in the domain $\Omega$ and $N^\partial$ boundary grid points on 
$\partial\Omega$.
Define
$$\mathbf u=\begin{pmatrix}
u_1&   \cdots& u_N             
            \end{pmatrix}^T,  \mathbf u^\partial=\begin{pmatrix}
u^\partial_1&   \cdots& u^\partial_{N^\partial}             
            \end{pmatrix}^T, \tilde{\mathbf u}=\begin{pmatrix}
u_1&   \cdots& u_N & u^\partial_1&  \cdots& u^\partial_{N^\partial}            
            \end{pmatrix}^T.$$
 A finite difference scheme can be written as 
 \begin{align*}
\mathcal L_h (\tilde{\mathbf u})_i=\sum_{j=1}^N b_{ij}u_j+\sum_{j=1}^{N^\partial}  b^\partial_{ij}u^\partial_j=&f_i,\quad 1\leq i\leq N,\\
  u^\partial_i=&g_i, \quad 1\leq i\leq N^\partial.  
 \end{align*}
 The matrix form is 
 \[\tilde L_h \tilde{\mathbf u}=\tilde{\mathbf f},\tilde L_h =\begin{pmatrix}
                       L_h & B^\partial\\
                       0 & I
                      \end{pmatrix}, \tilde{\mathbf u}=\begin{pmatrix}
                       \mathbf u\\ \mathbf u^\partial \end{pmatrix},
                       \tilde{\mathbf f}=\begin{pmatrix}
                       \mathbf f\\ \mathbf g \end{pmatrix}.
\]
The discrete maximum principle is 
\begin{equation}
\mathcal L_h (\tilde{\mathbf u})_i\leq 0, 1\leq i \leq N\Longrightarrow \max_i u_i\leq \max\{0, \max_i u_i^\partial\},  
\label{dmp}
\end{equation}
which implies 
\[\mathcal L_h (\tilde{\mathbf u})_i= 0, 1\leq i \leq N\Longrightarrow |u_i|\leq \max_i |u_i^\partial|.   \]

 The following result was proven in \cite{ciarlet1970discrete}:
\begin{theorem}
 A finite difference operator $\mathcal L_h$ satisfies the discrete maximum principle 
 \eqref{dmp} if  $\tilde L_h^{-1}\geq 0$ and  all row sums of $\tilde L_h$ are non-negative. 
 \end{theorem}
With the same $\bar L_h$ as defined in the previous section, it suffices to have $\bar{L}_h^{-1}\geq 0$, see  \cite{li2019monotonicity}:
\begin{theorem}
\label{theorem-fullandsmallmatrix}
 If $\bar{L}_h^{-1}\geq 0$, then $\tilde L_h^{-1}\geq 0 $ thus ${L}_h^{-1}\geq 0$. Moreover, if row sums of $\bar{L}_h$ are non-negative, then the finite difference operator $\mathcal L_h$ satisfies the discrete maximum principle.
\end{theorem}

Let $\mathbf 1$ be an abstract vector of the same shape as $\bar{\mathbf u}$ with all ones. For the $Q^2$  spectral element method, we have that $(\bar L_h \mathbf 1)_{i,j}=1$ if $(x_i, y_j)\in \partial \Omega$ and  $(\bar L_h\mathbf 1)_{i,j}=0$ if $(x_i, y_j)\in \Omega$, which implies the row sums of $\bar{L}_h$ are non-negative. Thus from now on, we only need to discuss the monotonicity of the matrix $\bar{L}_h$.

\subsection{Lorenz's sufficient condition for monotonicity}
\label{sec-lorenz-2}

\begin{defi}
Let $\mathcal N = \{1,2,\dots,n\}$. For $\mathcal N_1, \mathcal N_2 \subset \mathcal N$, we say a matrix $A$ of size $n\times n$ connects $\mathcal N_1$ with $\mathcal N_2$ if 
\begin{equation}
\forall i_0 \in \mathcal N_1, \exists i_r\in \mathcal N_2, \exists i_1,\dots,i_{r-1}\in \mathcal N \quad \mbox{s.t.}\quad  a_{i_{k-1}i_k}\neq 0,\quad k=1,\cdots,r.
\label{condition-connect}
\end{equation}
If perceiving $A$ as a directed graph adjacency matrix of vertices labeled by $\mathcal N$, then \eqref{condition-connect} simply means that there exists a directed path from any vertex in $\mathcal N_1$ to at least one vertex in $\mathcal N_2$.  
In particular, if $\mathcal N_1=\emptyset$, then any matrix $A$  connects $\mathcal N_1$ with $\mathcal N_2$.
\end{defi}

Given a square matrix $A$ and a column vector $\mathbf x$, we define
\[\mathcal N^0(A\mathbf x)=\{i: (A\mathbf x)_i=0\},\quad 
\mathcal N^+(A\mathbf x)=\{i: (A\mathbf x)_i>0\}.\]

Given a matrix $A=[a_{ij}]\in \mathbbm{R}^{n\times n}$, define its diagonal, off-diagonal, positive and negative off-diagonal parts as $n\times n$ matrices $A_d$, $A_a$, $A_a^+$, $A_a^-$:
\[(A_d)_{ij}=\begin{cases}
a_{ii}, & \mbox{if} \quad i=j\\
0, & \mbox{if} \quad  i\neq j
\end{cases}, \quad A_a=A-A_d,
\]
\[(A_a^+)_{ij}=\begin{cases}
a_{ij}, & \mbox{if} \quad a_{ij}>0,\quad i\neq j\\
0, & \mbox{otherwise}.
\end{cases}, \quad A_a^-=A_a-A^+_a.
\]

The following two results were proven in  \cite{lorenz1977inversmonotonie}. See also \cite{li2019monotonicity} for a detailed proof.
\begin{theorem}\label{thm2}
If $A\leq M_1M_2\cdots M_k L$ where $M_1, \cdots, M_k$ are nonsingular M-matrices and $L_a\leq 0$,  and there exists a nonzero vector $\mathbf e\geq 0$ such that \textcolor{black}{$A\mathbf e\geq 0$ and} one of the matrices $M_1, \cdots, M_k , L$ connects $\mathcal N^0(A\mathbf e)$ with $\mathcal N^+(A\mathbf e)$. Then $M_k^{-1}M_{k-1}^{-1}\cdots M_1^{-1} A$ is an M-matrix, thus $A$ is a product of $k+1$ nonsingular M-matrices and $A^{-1}\geq 0$. 
\end{theorem}
\begin{theorem}[Lorenz's condition] \label{thm3}
If $A^-_a$ has a decomposition: $A^-_a = A^z + A^s = (a_{ij}^z) + (a_{ij}^s)$ with $A^s\leq 0$ and $A^z \leq 0$, such that 
\begin{subequations}
 \label{lorenz-condition}
\begin{align}
& A_d + A^z \textrm{ is a nonsingular M-matrix},\label{cond1}\\ 
& A^+_a \leq A^zA^{-1}_dA^s \textrm{ or equivalently } \forall a_{ij} > 0 \textrm{ with } i \neq j, a_{ij} \leq \sum_{k=1}^n a_{ik}^za_{kk}^{-1}a_{kj}^s,\label{cond2}\\
& \exists \mathbf e \in \mathbbm{R}^n\setminus\{\mathbf 0\}, \mathbf e\geq 0 \textrm{ with $A\mathbf e \geq 0$ s.t. $A^z$ or $A^s$  connects $\mathcal N^0(A\mathbf e)$ with $\mathcal N^+(A\mathbf e)$.} \label{cond3}
\end{align}
\end{subequations}
Then $A$ is a product of two nonsingular M-matrices thus $A^{-1}\geq 0$.
\end{theorem}

\begin{Proposition} The matrix $L$ in Theorem \ref{thm2} must be an M-matrix.
\end{Proposition}

\begin{proof} Let $M^{-1} = M_k^{-1}M_{k-1}^{-1}...M_1^{-1}$,  following the proof of Theorem 7 in \cite{li2019monotonicity}, then $M^{-1}A\mathbf{e} \geq cA\mathbf{e}$ for some positive number $c$. Then $A\mathbf{e} \geq 0 \Rightarrow M^{-1}A\mathbf{e} \geq 0$. Now since $\mathbf{e} \geq 0$, $M^{-1}A \leq L \Rightarrow  0\leq (L-M^{-1}A)\mathbf e\Rightarrow M^{-1}A\mathbf{e} \leq L\mathbf{e}$ thus $L\mathbf{e} \geq 0$.

Assume $L$ connects $\mathcal{N}^0(A\mathbf{e})$ with $\mathcal{N}^+(A\mathbf{e})$. Since  $M^{-1}A\mathbf{e} \leq{} L\mathbf{e}$, $\mathcal{N}^0(L\mathbf{e}) \subseteq \mathcal{N}^0(A\mathbf{e})$ and $\mathcal{N}^+(A\mathbf{e}) \subseteq \mathcal{N}^+(L\mathbf{e})$, so $L$ also connects $\mathcal{N}^0(L\mathbf{e})$ with $\mathcal{N}^+(L\mathbf{e})$. 

Assume $M_i$ connects $\mathcal{N}^0(A\mathbf{e})$ with $\mathcal{N}^+(A\mathbf{e})$, following the proof of Theorem 7 in \cite{li2019monotonicity}, we have $M^{-1}A\mathbf{e} > 0$. Now $L$ trivially connects $\mathcal{N}^0(L\mathbf{e})$ with $\mathcal{N}^+(L\mathbf{e})$ since $L\mathbf{e} \geq M^{-1}A\mathbf{e} \Rightarrow{} L\mathbf{e} > 0$ and $\mathcal{N}^0(L\mathbf{e}) = \emptyset{}$. 

Then Theorem 6  in \cite{li2019monotonicity} applies to show $L$ is an M-matrix. 
\end{proof}

In practice, the condition \eqref{cond3} can be difficult to verify. For variational difference schemes, the vector $\mathbf e$ can be taken as $\mathbf 1$ consisting of all ones, then the condition \eqref{cond3} can be simplified. The following theorem was proven in \cite{li2019monotonicity}.

\begin{theorem}
 \label{newthm3}
Let $A$ denote the matrix representation of the variational difference scheme \eqref{nonhom-var-num3} with $Q^2$ basis  solving $-\nabla\cdot( a\nabla)u+cu=f$. 
Assume $A^-_a$ has a decomposition $A^-_a = A^z + A^s$ with $A^s\leq 0$ and $A^z \leq 0$.  
Then $A^{-1}\geq 0$ if the following are satisfied:
\begin{enumerate}
\item $(A_d+A^z)\mathbf 1\neq \mathbf 0$  and $(A_d+A^z)\mathbf 1\geq 0$;
\item $A^+_a \leq A^zA^{-1}_dA^s$;
 \item For $c(x,y)\geq 0$, either $A^z$ or $A^s$ has the same sparsity pattern as $A^-_a$. 
 If $c(x,y)>0$, then this condition can be removed.  
\end{enumerate}

\end{theorem}

\subsection{A relaxed Lorenz's condition}
\label{sec-lorenz-3}
In practice, both  \eqref{cond1} and  \eqref{cond2} impose mesh constraints for the $Q^2$ spectral element method on
non-uniform meshes. The condition    \eqref{cond1}
can be relaxed as the following:
\begin{theorem}[A relaxed Lorenz's condition] \label{thm3-r}
If $A^-_a$ has a decomposition: $A^-_a = A^z + A^s = (a_{ij}^z) + (a_{ij}^s)$ with $A^s\leq 0$ and $A^z \leq 0$,
and there exists a diagonal matrix $A_{d^*}\geq A_{d}$ such that 
\begin{subequations}
 \label{lorenz-condition-r}
\begin{align}
& A_d^* + A^z \textrm{ is a nonsingular M-matrix},\label{cond1-r}\\ 
& A^+_a \leq A^zA^{-1}_{d^*}A^s,\label{cond2-r} \\
& \exists \mathbf e \in \mathbbm{R}^n\setminus\{\mathbf 0\}, \mathbf e\geq 0 \textrm{ with $A\mathbf e \geq 0$ s.t. $A^z$ or $A^s$  connects $\mathcal N^0(A\mathbf e)$ with $\mathcal N^+(A\mathbf e)$.} \label{cond3-r}
\end{align}
\end{subequations}
Then $A$ is a product of two nonsingular M-matrices thus $A^{-1}\geq 0$.
\end{theorem}
\begin{proof}
It is straightforward that $A=A_d + A_a^++A^z+A^s \leq A_{d^*} + A^z +A^s+ A^zA_{d^*}^{-1}A^s= (A_{d^*} + A^z)(I + A_{d^*}^{-1}A^s)$.
By \eqref{cond3-r}, either $A_{d^*} + A^z$ or $I+A^{-1}_{d^*}A^s$ connects $\mathcal N^0(A\mathbf e)$ with $\mathcal N^+(A\mathbf e)$. 
By applying Theorem \ref{thm2} for the case $k=1$, $M_1=A_{d^*}+A^z$ and $L=I+A^{-1}_{d^*}A^s$, we get $A^{-1}\geq 0$. 
\end{proof}

\begin{rmk}
Since $A_{d}\leq A_{d^*}$, only \eqref{cond1-r} is more relaxed than \eqref{cond1}, and  \eqref{cond2-r} is more stringent than \eqref{cond2}.
However, we will show in next section that it is possible to construct $A_{d^*}$ such that \eqref{cond2} and \eqref{cond2-r} impose identical mesh constraints.
\end{rmk}

With Theorem \ref{rowsumcondition-thm},  combining Theorem \ref{thm3-r} and Theorem \ref{newthm3}, we have:
\begin{theorem}
 \label{newthm3-r}
Let $A$ denote the matrix representation of the variational difference scheme \eqref{nonhom-var-num3} with $Q^2$ basis  solving $-\nabla\cdot( a\nabla)u+cu=f$. 
Assume $A^-_a$ has a decomposition $A^-_a = A^z + A^s$ with $A^s\leq 0$ and $A^z \leq 0$
and there exists a diagonal matrix $A_{d^*}\geq A_{d}$.  
Then $A^{-1}\geq 0$ if the following are satisfied:
\begin{enumerate}
\item $(A_{d^*}+A^z)\mathbf 1\neq \mathbf 0$  and $(A_{d^*}+A^z)\mathbf 1\geq 0$;
\item $A^+_a \leq A^zA^{-1}_{d^*}A^s$;
 \item For $c(x,y)\geq 0$, either $A^z$ or $A^s$ has the same sparsity pattern as $A^-_a$. 
 If $c(x,y)>0$, then this condition can be removed.  
\end{enumerate}

\end{theorem}

\section{Monotonicity of $Q^2$ spectral element method on quasi-uniform meshes}
\label{sec:main}

The $Q^2$ spectral element method
has been proven  monotone on a uniform mesh for Laplacian operator without any  mesh constraints \cite{li2019monotonicity}.
In this section, we will discuss its monotonicity for the Laplacian operator on quasi-uniform meshes. 
The discussion in this section can be easily extended to more general cases such as $\mathcal L u=-\Delta u+cu$ and Neumann boundary conditions. For simplicity, we only discuss the Laplacian case $\mathcal L u=-\Delta u$ and Dirichlet boundary conditions. 

 Consider a grid $(x_i,y_j)$ ($i, j= 0,1,\dots, n+1$) for a rectangular domain $[0,1]\times[0,1]$
 where $n$ must be odd and $i,j=0,n+1$ correspond to boundary points.
Let $u_{ij}$ denote the numerical solution at $(x_i, y_j)$.
Let $\bar{\mathbf u}$ denote an abstract vector consisting of $u_{ij}$ for $i,j=0,1,2,\cdots,n,n+1$.
Let $\bar{\mathbf f}$ denote an abstract vector consisting of $f_{ij}$ for $i,j=1,2,\cdots,n$ and the boundary condition $g$ at the boundary grid points. 
Then the matrix vector representation of  \eqref{nonhom-var-num3} with $Q^2$ basis is
$\bar L_h \bar{\mathbf u} =\bar{\mathbf f}$.

The focus of this section is to show $\bar L_h^{-1}\geq 0$ under suitable mesh constraints for quasi-uniform meshes. 
 Moreover, it is straightforward to verify that $(\bar L_h  \mathbf 1 )_{i,j}=0$ for interior points $(x_i, y_j)$ and $(\bar L_h  \mathbf 1 )_{i,j}=1$ for boundary points $(x_i, y_j)$. 
 Thus by Section \ref{sec-dmp}, the scheme also satisfies the discrete maximum principle.

For simplicity, in the rest of this section we use $A$ to denote the matrix $\bar L_h$ and 
let $\mathcal A$ be the linear operator corresponding to the matrix $A$. 
For convenience, we can also regard the abstract vector $\bar{\mathbf u}$ as a matrix of size $(n+2)\times(n+2)$. 
Then by our notation, the mapping $\mathcal A: \mathbbm R^{(n+2)\times(n+2)}\rightarrow \mathbbm R^{(n+2)\times(n+2)}$ is given as $\mathcal A( \bar{\mathbf u} )_{i,j}:=
(\bar L_h \bar{\mathbf u} )_{i,j}$. 
\subsection{The scheme in two dimensions}\label{q2_scheme}
     \begin{figure}[htbp]
   \centering
 \subfigure[Mesh length definitions for four adjacent $Q^2$ elements. ]{\includegraphics{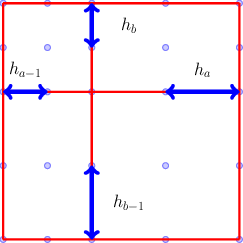}}
 \hspace{.6in}
 \subfigure[The four distinct point types. ]{
\includegraphics{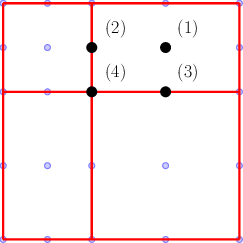} }
\caption{A non-uniform mesh for $Q^2$ spectral element method. Each edge in a cell has length $2h$.}
\label{mesh}
 \end{figure}

 For boundary points $(x_i, y_j)\in\partial \Omega$, the scheme is $\mathcal A( \bar{\mathbf u} )_{i,j}:=u_{i,j}=g_{i,j}$. The scheme for interior grid points $(x_i, y_j)\in \Omega$ on a non-uniform mesh can be given on  four distinct types of points shown in  Figure \ref{mesh} (b).  For simplicity, from now on, we will use {\it edge center (2)} to denote an interior edge center for an edge  parallel to the y-axis, and {\it edge center (3)} to denote an interior edge center for an edge  parallel to the x-axis. The scheme at an interior grid point is given as $\mathcal A( \bar{\mathbf u} )_{i,j}=f_{i,j}$ with
\begin{align}
\label{q2-scheme-nonuniform}       \mathcal A( \bar{\mathbf u} )_{i,j}:=& \frac{2h_a^2+2h_b^2}{h_a^2h_b^2}u_{i,j}-\left(\frac{1}{h_a^2}u_{i+1,j} + \frac{1}{h_a^2}u_{i-1,j} + \frac{1}{h_b^2}u_{i,j+1} + \frac{1}{h_b^2}u_{i,j-1}\right) \\ &\text{if $(x_i,y_j)$ is a cell center} \notag;
\\ 
              \mathcal A( \bar{\mathbf u} )_{i,j}:=  & \frac{7h_b^2+4h_ah_{a-1}}{2h_ah_{a-1}h_b^2}u_{i,j} -\frac{4}{h_a(h_a + h_{a-1})}u_{i+1,j}-\frac{4}{h_{a-1}(h_a + h_{a-1})}u_{i-1,j} \notag \\
              &-\frac{1}{h_b^2}u_{i,j+1} - \frac{1}{h_b^2}u_{i,j-1} +\frac{1}{2h_a(h_a+h_{a-1})}u_{i+2,j} + \frac{1}{2h_{a-1}(h_a+h_{a-1})}u_{i-2,j}, \notag\\ 
              & \text{if $(x_i,y_j)$ is edge center (2);} \notag
\\ 
            \mathcal A( \bar{\mathbf u} )_{i,j}:= &    \frac{7h_a^2+4h_bh_{b-1}}{2h_bh_{b-1}h_a^2}u_{i,j}- \frac{4}{h_b(h_b + h_{b-1})}u_{i,j+1}-\frac{4}{h_{b-1}(h_b + h_{b-1})}u_{i,j-1} \notag\\
            & -\frac{1}{h_a^2}u_{i+1,j} - \frac{1}{h_a^2}u_{i-1,j}+ \frac{1}{2h_{b}(h_b+h_{b-1})}u_{i,j+2} + \frac{1}{2h_{b-1}(h_b+h_{b-1})}u_{i,j-2}, \notag\\ 
            & \text{if $(x_i,y_j)$ is edge center (3);} \notag
\\ 
            \mathcal A( \bar{\mathbf u} )_{i,j}:=  & \frac{7h_ah_{a-1}+7h_bh_{b-1}}{2h_ah_{a-1}h_bh_{b-1}}u_{i,j}-\left[ \frac{4}{h_a(h_a + h_{a-1})}u_{i+1,j}+\frac{4}{h_{a-1}(h_a + h_{a-1})}u_{i-1,j}\right. 
         \notag   \\&  \left. + \frac{4}{h_b(h_b + h_{b-1})}u_{i,j+1}+  \frac{4}{h_{b-1}(h_b + h_{b-1})}u_{i,j-1}\right] + \frac{1}{2h_a(h_a+h_{a-1})}u_{i+2,j} 
         \notag   \\ &+\frac{1}{2h_{a-1}(h_a+h_{a-1})}u_{i-2,j} + \frac{1}{2h_{b}(h_b+h_{b-1})}u_{i,j+2} + \frac{1}{2h_{b-1}(h_b+h_{b-1})}u_{i,j-2},
          \notag  \\& \text{if $(x_i,y_j)$ is an interior knot.} \notag
    \end{align}

For a uniform mesh $h_{a}=h_{a-1}=h_b=h_{b-1}=h$, the scheme reduces to \eqref{Q2-laplacian}.

\subsection{The Decomposition of $A_a^-$}

Next, by the same notations defined in Section \ref{sec-lorenz-2}, we will decompose the matrix $A = A_d + A_a^- + A_a^+$ and  $A_a^- = A^z + A^s$  to verify Theorem \ref{newthm3}. We will use $\mathcal A_a^-$, $\mathcal A_a^+$, $\mathcal A^z$ and $\mathcal A^s$ to denote 
linear operators for corresponding  matrices.
First, for the diagonal part we have 
\begin{align*}
           \mathcal A_d( \bar{\mathbf u} )_{i,j} & = u_{i,j }, \quad \text{if $(x_i,y_j)$ is a boundary point;}\\
           \mathcal A_d( \bar{\mathbf u} )_{i,j} & = \frac{2h_a^2 + 2h_b^2}{h_a^2h_b^2}u_{i,j}, \quad \text{if $(x_i,y_j)$ is a cell center;}\\
  \mathcal A_d( \bar{\mathbf u} )_{i,j}& = \frac{7h_b^2 + 4h_ah_{a-1}}{2h_ah_{a-1}h_b^2}u_{i,j},\quad \text{if $(x_i,y_j)$ is edge center (2);}\\
   \mathcal A_d( \bar{\mathbf u} )_{i,j} &= \frac{7h_a^2 + 4h_bh_{b-1}}{2h_bh_{b-1}h_a^2}u_{i,j},  \quad\text{if $(x_i,y_j)$ is edge center (3);}
\\ 
   \mathcal A_d( \bar{\mathbf u} )_{i,j} & = \frac{7h_bh_{b-1}+ 7h_ah_{a-1}}{2h_ah_{a-1}h_bh_{b-1}}u_{i,j}, \quad \text{if $(x_i,y_j)$ is an interior  knot.}
    \end{align*}
Notice that for a boundary point  $(x_i, y_j)\in\partial \Omega$ we have $ \mathcal A( \bar{\mathbf u} )_{i,j} =\mathcal A_d( \bar{\mathbf u} )_{i,j} = u_{i,j },$
thus for off-diagonal parts, we only need to look at the interior grid points.    
 For positive off-diagonal entries, we have
\begin{align*}
\mathcal A^+_a( \bar{\mathbf u} )_{i,j}  =& 0, \quad \text{if $(x_i,y_j)$ is a cell center;}
\\ 
   \mathcal A^+_a( \bar{\mathbf u} )_{i,j} =& \frac{1}{2h_a(h_a+h_{a-1})}u_{i+2,j} + \frac{1}{2h_{a-1}(h_a+h_{a-1})}u_{i-2,j}, \quad\text{edge center (2);}
 \\ 
    \mathcal A^+_a( \bar{\mathbf u} )_{i,j} = &\frac{1}{2h_{b}(h_b+h_{b-1})}u_{i,j+2} + \frac{1}{2h_{b-1}(h_b+h_{b-1})}u_{i,j-2}, \quad\text{edge center (3);}\\ 
   \mathcal A^+_a( \bar{\mathbf u} )_{i,j} =& \frac{1}{2h_a(h_a+h_{a-1})}u_{i+2,j} + \frac{1}{2h_{a-1}(h_a+h_{a-1})}u_{i-2,j} + \frac{1}{2h_{b}(h_b+h_{b-1})}u_{i,j+2}\\&  + \frac{1}{2h_{b-1}(h_b+h_{b-1})}u_{i,j-2},\quad \text{if $(x_i,y_j)$ is an interior knot.} 
    \end{align*}

Then we perform a decomposition $A_a^- = A^z + A^s$, which depends on two constants $0<\epsilon_1 \leq 1$ and $0<\epsilon_2 \leq 1$. 
\begin{align*}
    \mathcal A^z(\bar{\mathbf u})_{i,j}  = &-\epsilon_1\left( \frac{1}{h_a^2}u_{i+1,j} + \frac{1}{h_a^2}u_{i-1,j} + \frac{1}{h_b^2}u_{i,j+1} + \frac{1}{h_b^2}u_{i,j-1}\right), \quad \text{if $(x_i,y_j)$ is a cell center;}
\\ 
      \mathcal A^z(\bar{\mathbf u})_{i,j} =& -\epsilon_1\left( \frac{1}{h_b^2}u_{i,j+1} + \frac{1}{h_b^2}u_{i,j-1}\right) -\epsilon_2\left[\frac{4}{h_a(h_a + h_{a-1})}u_{i+1,j}+\frac{4}{h_{a-1}(h_a + h_{a-1})}u_{i-1,j}\right],\\ &\text{if $(x_i,y_j)$ is edge center (2);}
\\ 
      \mathcal A^z(\bar{\mathbf u})_{i,j} =& -\epsilon_1\left( \frac{1}{h_a^2}u_{i+1,j} + \frac{1}{h_a^2}u_{i-1,j}\right) -\epsilon_2\left[ \frac{4}{h_b(h_b + h_{b-1})}u_{i,j+1}+\frac{4}{h_{b-1}(h_b + h_{b-1})}u_{i,j-1}\right],\\ &\text{if $(x_i,y_j)$ is edge center (3);}
\\ 
      \mathcal A^z(\bar{\mathbf u})_{i,j} =& -\epsilon_2\left[\frac{4}{h_a(h_a + h_{a-1})}u_{i+1,j}+\frac{4}{h_{a-1}(h_a + h_{a-1})}u_{i-1,j} \right. 
      \\
&     \left. +\frac{4}{h_b(h_b + h_{b-1})}u_{i,j+1}+\frac{4}{h_{b-1}(h_b + h_{b-1})}u_{i,j-1}\right],\quad \text{if $(x_i,y_j)$ is an interior knot.}
    \end{align*}
Notice that $A^z$ defined above has exactly the same sparsity pattern as $A_a^-$ for $0<\epsilon_1 \leq 1$ and $0<\epsilon_2 \leq 1$. 
Let $A^s = A_a^- - A^z$ then $A^s\leq 0$. 


\subsection{Mesh constraints for $A^zA_d^{-1}A^s\geq A^+_a$} \label{azadas}

In order to verify $A^zA_d^{-1}A^s\geq A_a^+$, we only need to discuss nonzero entries in the output of $\mathcal A_a^+(\bar{\mathbf u})$ since $A^zA_d^{-1}A^s\geq 0$.

   \begin{figure}[h]
\subfigure[Four red dots denote non-zero entry locations in $\mathcal{A}_a^+(\bar{\mathbf u})_{i,j}$]{\includegraphics[scale=1]{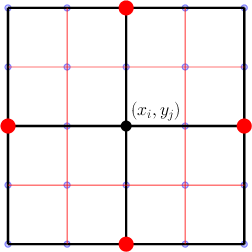}}
\subfigure[Stencil  of $\mathcal{A}^z(\bar{\mathbf u})_{i,j}$. ]{\includegraphics[scale=1]{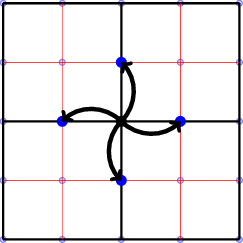}}
\subfigure[Stencil  of $\mathcal{A}^z \mathcal A_d^{-1} \mathcal{A}^s(\bar {\mathbf u})_{i,j} $. ]{\includegraphics[scale=1]{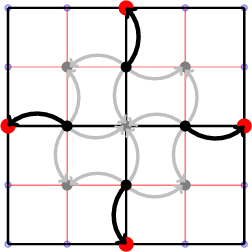}}
\caption{Stencil   of operators  at an  interior knot $(x_i, y_j)$. The four red dots are the locations/entries where $\mathcal{A}_a^+(\bar {\mathbf u})_{i,j}$ are nonzero. Gray nodes in (c) represent positive entries that can be discarded for the purposes of verifying \eqref{cond2-r}. The mesh is illustrated as a uniform one only for simplicity.}
\label{aa_plus_knot}
 \end{figure}

First consider the case that $(x_i,y_j)$ is an interior knot. Figure \ref{aa_plus_knot} (a) shows the positive coefficients in the output of $\mathcal A_a^+(\bar{\mathbf u})_{ij}$ at a knot $(x_i,y_j)$. Figure  \ref{aa_plus_knot} (b) shows  the stencil of $\mathcal A^z(\bar{\mathbf{u}})_{ij}$. Thus $\mathcal A^z(\bar{\mathbf{u}})$ acting as an operator on 
$[\mathcal{A}_d^{-1} \mathcal{A}^s](\bar {\mathbf u})$ at a knot is:
\[\resizebox{\textwidth}{!} 
{$     \begin{gathered}
    [\mathcal{A}^z\mathcal{A}_d^{-1}\mathcal{A}^s](\bar {\mathbf u})_{i,j} =-4\epsilon_2\left[\frac{1}{h_a(h_{a-1}+h_a)}[\mathcal{A}_d^{-1}\mathcal{A}^s](\bar {\mathbf u})_{i+1,j}+  \frac{1}{h_{a-1}(h_{a-1}+h_a)}[\mathcal{A}_d^{-1}\mathcal{A}^s](\bar {\mathbf u})_{i-1,j}\right.\\
    \left.+  \frac{1}{h_b(h_{b-1}+h_b)}[\mathcal{A}_d^{-1}\mathcal{A}^s](\bar {\mathbf u})_{i,j+1}+\frac{1}{h_{b-1}(h_{b-1}+h_b)}[\mathcal{A}_d^{-1}\mathcal{A}^s](\bar {\mathbf u})_{i,j-1}\right].
       \end{gathered}$}
\]
In the expression above, the output of the operator $\mathcal A^z(\bar{\mathbf{u}})_{ij}$  are at interior edge centers as shown in Figure  \ref{aa_plus_knot} (b). Hence $[\mathcal{A}_d^{-1}\mathcal{A}^s]$ will act on these edge centers with the mesh lengths corresponding to Figure \ref{mesh}. Carefully considering the mesh lengths and operations of $\mathcal{A}_d^{-1}$ at these points gives:
\[\resizebox{\textwidth}{!} 
{$     \begin{gathered}
    [\mathcal{A}^z\mathcal{A}_d^{-1}\mathcal{A}^s](\bar{\mathbf u} )_{i,j} = -4\epsilon_2\left [\frac{1}{h_a(h_{a-1}+h_a)}\frac{2h_bh_{b-1}h_a^2}{7h_a^2+4h_bh_{b-1}}\mathcal{A}^s(\bar{\mathbf u} )_{i+1,j} \right.
   \\ +\frac{1}{h_{a-1}(h_{a-1}+h_a)}\frac{2h_bh_{b-1}h_{a-1}^2}{7h_{a-1}^2+4h_bh_{b-1}}\mathcal{A}^s(\bar{\mathbf u} )_{i-1,j}+\frac{1}{h_b(h_{b-1}+h_b)}\frac{2h_ah_{a-1}h_b^2}{7h_b^2+4h_ah_{a-1}}\mathcal{A}^s(\bar{\mathbf u} )_{i,j+1}
    \\  \left.+ \frac{1}{h_{b-1}(h_{b-1}+h_b)}\frac{2h_ah_{a-1}h_{b-1}^2}{7h_{b-1}^2+4h_ah_{a-1}}\mathcal{A}^s(\bar{\mathbf u} )_{i,j-1}\right],\quad \text{if $(x_i,y_j)$ is an interior knot.}
    \end{gathered}$}
\]
Next consider
 the effect of $\mathcal{A}^s(\bar{\mathbf{u}})$ operator which has the same sparsity pattern as $\mathcal{A}^z(\bar{\mathbf{u}})$. Figure  \ref{aa_plus_knot} (c) shows the stencil of $    [\mathcal{A}^z\mathcal{A}_d^{-1}\mathcal{A}^s](\bar {\mathbf u})_{i,j} $ for an interior knot. Recall that $A^z \leq 0$, $A^s \leq 0$, and $A_d^{-1} \geq 0$, thus we have $A^zA_d^{-1}A^s\geq0$. So we only need to compare the outputs of  $[\mathcal{A}^z\mathcal{A}_d^{-1}\mathcal{A}^s](\bar {\mathbf u})_{i,j} $ and $\mathcal{A}_a^+(\bar {\mathbf u})_{i,j}$ at nonzero entries of  $\mathcal{A}_a^+(\bar {\mathbf u})_{i,j}$, i.e., the four red dots in Figure \ref{aa_plus_knot}  (a) and Figure  \ref{aa_plus_knot} (c).  

 Thus we only need coefficients of $u_{i+2,j}, u_{i-2,j}, u_{i,j+2},$ and $u_{i,j-2}$ in the final expression of $[\mathcal{A}^z\mathcal{A}_d^{-1}\mathcal{A}^s](\mathbf u)_{i,j}$,
 which are found to be   
 \begin{itemize}
     \item[] $u_{i+2,j}:\quad{}4\epsilon_2(1-\epsilon_1)\frac{1}{h_a(h_{a-1}+h_a)}\frac{2h_bh_{b-1}h_a^2}{7h_a^2+4h_bh_{b-1}}\frac{1}{h_a^2}$
     
     \item[] $u_{i-2,j}:\quad{}4\epsilon_2(1-\epsilon_1)\frac{1}{h_{a-1}(h_{a-1}+h_a)}\frac{2h_bh_{b-1}h_{a-1}^2}{7h_{a-1}^2+4h_bh_{b-1}}\frac{1}{h_{a-1}^2}$
     
     \item[] $u_{i,j+2}:\quad{}4\epsilon_2(1-\epsilon_1)\frac{1}{h_b(h_{b-1}+h_b)}\frac{2h_ah_{a-1}h_b^2}{7h_b^2+4h_ah_{a-1}}\frac{1}{h_b^2}$
     
     \item[] $u_{i,j-2}:\quad{}4\epsilon_2(1-\epsilon_1)\frac{1}{h_{b-1}(h_{b-1}+h_b)}\frac{2h_ah_{a-1}h_{b-1}^2}{7h_{b-1}^2+4h_ah_{a-1}}\frac{1}{h_{b-1}^2}$
     
     \item[]
 \end{itemize}

In order to maintain $A_a^+ \leq A^zA_d^{-1}A^s$,
by comparing to the coefficients of $u_{i+2,j}$ for $\mathcal A_a^+(\bar{\mathbf{u}})$, we obtain a mesh constraint $
    4\epsilon_2(1-\epsilon_1)\frac{2h_bh_{b-1}}{7h_a^2+4h_bh_{b-1}} \geq \frac{1}{2}.$
Similar constraints are obtained by comparing other coefficients at $u_{i,j\mp2}$ and $u_{i-2, j}$. 
Define 
$$\ell(\epsilon_1, \epsilon_2) =4\epsilon_2(1-\epsilon_1).$$ Then the following constraints are sufficient for $\mathcal A_a^+(\bar{\mathbf{u}})$ to be controlled by
$\mathcal A^z\mathcal A_d^{-1}\mathcal A^s(\bar{\mathbf{u}})$ at an interior knot:
\begin{subequations}
\label{meshconstraint-1}
\begin{equation}
    h_ah_{a-1} \geq\frac{7}{4\ell-4}\max\{h_b^2,h_{b-1}^2\}, 
   \quad h_bh_{b-1} \geq \frac{7}{4\ell-4}\max\{h_a^2,h_{a-1}^2\}.
\end{equation}

Second, we need to discuss the case when $(x_i,y_j)$ is an interior edge center. Without loss of generality, assume $(x_i, y_j)$ is an interior edge center of an edge parallel to the y-axis. Then similar to the interior knot case, the output coefficients of $[\mathcal{A}^z\mathcal{A}_d^{-1}\mathcal{A}^s](\bar{\mathbf u})_{i,j}$ at the relevant non-zero entries of $\mathcal A_a^+(\bar{\mathbf u})_{i,j}$ are:
 
 \begin{itemize}
     \item[] $u_{i+2,j}:\quad{}4\epsilon_2(1-\epsilon_1)\frac{1}{h_a(h_{a-1}+h_a)}\frac{h_a^2h_b^2}{2h_a^2+2h_b^2}\frac{1}{h_a^2}$
     
     \item[] $u_{i-2,j}:\quad{}4\epsilon_2(1-\epsilon_1)\frac{1}{h_{a-1}(h_{a-1}+h_a)}\frac{h_{a-1}^2h_b^2}{2h_{a-1}^2+2h_b^2}\frac{1}{h_{a-1}^2}$
 \end{itemize}

By comparing with coefficients of $\mathcal A_a^+(\bar{\mathbf u})_{i,j}$, we get
$\frac{h_b^2}{h_a^2+h_b^2} \geq \frac{1}{\ell}, \quad \frac{h_b^2}{h_{a-1}^2+h_b^2} \geq \frac{1}{\ell}.$
To ensure $\mathcal A_a^+(\bar{\mathbf{u}})$ is controlled by
$\mathcal A^z\mathcal A_d^{-1}\mathcal A^s(\bar{\mathbf{u}})$ at  edge centers, it suffices to have:
\begin{equation}
    min\{h_a,h_{a-1}\} \geq \sqrt{\frac{1}{\ell-1}}max\{h_b,h_{b-1}\},\quad
    min\{h_b,h_{b-1}\} \geq \sqrt{\frac{1}{\ell-1}}max\{h_a,h_{a-1}\}.
\end{equation}
\end{subequations}

Note that $\mathcal A_a^+(\bar{\mathbf{u}})_{i,j} = 0$ if $(x_i,y_j)$ is a cell center. Since $\mathcal A^z\mathcal A_d^{-1}\mathcal A^s(\bar{\mathbf{u}}) \geq 0$, there is no mesh constraint to enforce the inequality at cell centers.

\subsection{Mesh constraints for $A_d + A^z$ being an M-matrix}\label{l_calc}
Let $\mathcal B=\mathcal{A}_d + \mathcal{A}^z$. Then $\mathcal B({\mathbf 1})_{i,j} = 1$ for a boundary point $(x_i,y_j)$. For interior points, we have:
\begin{align*}
   \mathcal B({\mathbf 1})_{i,j}& = -\epsilon_1\left(\frac{1}{h_a^2} + \frac{1}{h_a^2} + \frac{1}{h_b^2}+ \frac{1}{h_b^2}\right) + \frac{2h_a^2 + 2h_b^2}{h_a^2h_b^2}=(1-\epsilon{}_1)\frac{2h_a^2+2h_b^2}{h_a^2h_b^2}, \quad\text{ cell center;}
\\  \mathcal B({\mathbf 1})_{i,j} &= -\epsilon_1\left(\frac{1}{h_b^2} + \frac{1}{h_b^2} \right) -\epsilon_2\left[\frac{4}{h_a(h_a + h_{a-1})} +\frac{4}{h_{a-1}(h_a + h_{a-1})} \right] + \frac{7h_b^2 + 4h_ah_{a-1}}{2h_ah_{a-1}h_b^2}\\
&= (1-\epsilon{}_1)\frac{2}{h_b^2} + (1-\frac{8}{7}\epsilon{}_2)\frac{7}{2h_ah_{a-1}}, \quad\text{edge center (2);}  \\    \mathcal B({\mathbf 1})_{i,j} &= -\epsilon_1\left(\frac{1}{h_a^2}  + \frac{1}{h_a^2} \right) -\epsilon_2\left[\frac{4}{h_b(h_b + h_{b-1})}
    +\frac{4}{h_{b-1}(h_b + h_{b-1})} \right]+\frac{7h_a^2 + 4h_bh_{b-1}}{2h_bh_{b-1}h_a^2}\\
&= (1-\epsilon{}_1)\frac{2}{h_a^2} + (1-\frac{8}{7}\epsilon{}_2)\frac{7}{2h_bh_{b-1}},\quad\text{edge center (3);}\\ 
 \mathcal B({\mathbf 1})_{i,j} &= -\epsilon_2\left[\frac{4}{h_a(h_a + h_{a-1})} +\frac{4}{h_{a-1}(h_a + h_{a-1})}+ \frac{4}{h_b(h_b + h_{b-1})} +\frac{4}{h_{b-1}(h_b + h_{b-1})} \right]\\
 &+\frac{7h_bh_{b-1}+ 7h_ah_{a-1}}{2h_ah_{a-1}h_bh_{b-1}}= (1-\frac{8}{7}\epsilon{}_2)\frac{7h_bh_{b-1}+ 7h_ah_{a-1}}{2h_ah_{a-1}h_bh_{b-1}},\quad \text{interior knot.}
    \end{align*}

Notice that larger values of  $\ell$ give better mesh constraints in \eqref{meshconstraint-1}. And we have $\sup_{0<\epsilon_1, \epsilon_2\leq 1} \ell(\epsilon_1, \epsilon_2)=\sup_{0<\epsilon_1, \epsilon_2\leq 1} 4\epsilon_2(1-\epsilon_1) =4.$
In order to apply Theorem \ref{rowsumcondition-thm} for  $A_d + A^z$ be an M-matrix,  we need $ [\mathcal{A}_d + \mathcal{A}^z]({\mathbf 1})\geq 0$. This is true if and only if $\epsilon_1 \leq 1$ and $\epsilon_2 \leq \frac{7}{8}$, which only give $\sup_{0<\epsilon_1 \leq 1, 0<\epsilon_2 \leq \frac{7}{8}}\ell(\epsilon_1, \epsilon_2)= 3.5$.

\subsection{Improved mesh constraints by the relaxed Lorenz's condition}
To get a better mesh constraint, the constraint on $\epsilon_2$ can be relaxed so that the value of $\ell(\epsilon_1, \epsilon_2)$ can be improved. 
One observation from Section \ref{azadas} is that the value of $\mathcal A_d(\bar{\mathbf u})_{i,j}$ for $(x_i,y_j)$ being a knot is not used for verifying $A_a^+ \leq A^zA_d^{-1}A^s$ (for both interior knots and edge centers).  To this end, we define a new diagonal matrix $A_{d^*}$, which is different from $A_d$ only at the interior knots.
\begin{align*}
           \mathcal A_{d^*}( \bar{\mathbf u} )_{i,j} & = u_{i,j }= \mathcal A_{d}( \bar{\mathbf u} )_{i,j} , \quad \text{if $(x_i,y_j)$ is a boundary point;}\\
           \mathcal A_{d^*}( \bar{\mathbf u} )_{i,j} & = \frac{2h_a^2 + 2h_b^2}{h_a^2h_b^2}u_{i,j}= \mathcal A_{d}( \bar{\mathbf u} )_{i,j} , \quad \text{if $(x_i,y_j)$ is a cell center;}\\
  \mathcal A_{d^*}( \bar{\mathbf u} )_{i,j}& = \frac{7h_b^2 + 4h_ah_{a-1}}{2h_ah_{a-1}h_b^2}u_{i,j}= \mathcal A_{d}( \bar{\mathbf u} )_{i,j} ,\quad\text{ edge center (2);}\\
   \mathcal A_{d^*}( \bar{\mathbf u} )_{i,j} &= \frac{7h_a^2 + 4h_bh_{b-1}}{2h_bh_{b-1}h_a^2}u_{i,j}= \mathcal A_{d}( \bar{\mathbf u} )_{i,j} ,  \quad\text{edge center (3);}
\\ 
   \mathcal A_{d^*}( \bar{\mathbf u} )_{i,j} & =\frac{8h_bh_{b-1}+ 8h_ah_{a-1}}{2h_ah_{a-1}h_bh_{b-1}}u_{i,j}\neq \mathcal A_{d}( \bar{\mathbf u} )_{i,j} , \quad \text{if $(x_i,y_j)$ is an interior  knot.}
    \end{align*}    
Since the values of $\mathcal A_d(\bar{\mathbf u})_{i,j}$ for $(x_i,y_j)$ being a knot is not involved in Section  \ref{azadas}, the same discussion in Section  \ref{azadas} also holds for verifying $A_a^+ \leq A^zA_{d^*}^{-1}A^s$. Namely, under mesh constraints \eqref{meshconstraint-1}, we also have $A_a^+ \leq A^zA_{d^*}^{-1}A^s$.

Let $B^*=A_{d^*} + A^z$, then the row sums of $B^*$ are:
\begin{align*}
    \mathcal B^*({\mathbf 1})_{i,j}& = 1, \quad  \text{if $(x_i,y_j)$ is a boundary point;} \\
    \mathcal B^*({\mathbf 1})_{i,j}& = -\epsilon_1\left(\frac{1}{h_a^2} + \frac{1}{h_a^2} + \frac{1}{h_b^2}+ \frac{1}{h_b^2}\right) + \frac{2h_a^2 + 2h_b^2}{h_a^2h_b^2}=(1-\epsilon{}_1)\frac{2h_a^2+2h_b^2}{h_a^2h_b^2},\text{cell center;}
\\  \mathcal B^*({\mathbf 1})_{i,j} &= -\epsilon_1\left(\frac{1}{h_b^2} + \frac{1}{h_b^2} \right) -\epsilon_2\left[\frac{4}{h_a(h_a + h_{a-1})} +\frac{4}{h_{a-1}(h_a + h_{a-1})} \right] + \frac{7h_b^2 + 4h_ah_{a-1}}{2h_ah_{a-1}h_b^2}\\
&= (1-\epsilon{}_1)\frac{2}{h_b^2} + (1-\frac{8}{7}\epsilon{}_2)\frac{7}{2h_ah_{a-1}},  \quad\text{edge center (2);}\\
    \mathcal B^*({\mathbf 1})_{i,j} &= -\epsilon_1\left(\frac{1}{h_a^2}  + \frac{1}{h_a^2} \right) -\epsilon_2\left[\frac{4}{h_b(h_b + h_{b-1})}
    +\frac{4}{h_{b-1}(h_b + h_{b-1})} \right]+\frac{7h_a^2 + 4h_bh_{b-1}}{2h_bh_{b-1}h_a^2}\\
&= (1-\epsilon{}_1)\frac{2}{h_a^2} + (1-\frac{8}{7}\epsilon{}_2)\frac{7}{2h_bh_{b-1}}, \quad\text{edge center (3);}
\\ \mathcal B^*({\mathbf 1})_{i,j} &= -\epsilon_2\left[\frac{4}{h_a(h_a + h_{a-1})} +\frac{4}{h_{a-1}(h_a + h_{a-1})}+ \frac{4}{h_b(h_b + h_{b-1})} +\frac{4}{h_{b-1}(h_b + h_{b-1})} \right]\\
&+\frac{8h_bh_{b-1}+ 8h_ah_{a-1}}{2h_ah_{a-1}h_bh_{b-1}}
    = (1-\epsilon{}_2)\frac{8h_bh_{b-1}+ 8h_ah_{a-1}}{2h_ah_{a-1}h_bh_{b-1}}, \quad\text{ interior knot.}
    \end{align*}

Now $ [\mathcal{A}_{d^*} + \mathcal{A}^z]({\mathbf 1})_{i,j}\geq 0$ at cell centers and knots is true if and only if $\epsilon_1 \leq 1$ and $\epsilon_2 \leq 1$. 

Next, we will show that the mesh constraints \eqref{meshconstraint-1} with $0<\epsilon_1\leq \frac12$ and $\epsilon_2=1$ are sufficient to ensure 
$ [\mathcal{A}_{d^*} + \mathcal{A}^z]({\mathbf 1})_{i,j}\geq 0$ at edge centers. 
We have $0<\epsilon_1\leq \frac12, \epsilon_2=1\Longrightarrow  2\leq \ell<4\Longrightarrow \frac{7}{4\ell -4}\geq \frac{1}{\ell}.$ The mesh constraints \eqref{meshconstraint-1} imply that $h_a h_{a-1} \geq \frac{7}{4\ell -4}h_b^2\geq \frac{1}{\ell} h_b^2$, thus 
 \begin{align*}
& (1-\epsilon{}_1)\frac{2}{h_b^2} + (1-\frac{8}{7}\epsilon{}_2)\frac{7}{2h_ah_{a-1}}=(1-\epsilon{}_1)\frac{2}{h_b^2} -\frac{1}{2}\frac{1}{h_ah_{a-1}}
= \frac12 \left[ \frac{\ell}{h_b^2}-\frac{1}{h_a h_{a-1}}\right]\geq 0.
 \end{align*}
 Similarly, $(1-\epsilon{}_1)\frac{2}{h_a^2} + (1-\frac{8}{7}\epsilon{}_2)\frac{7}{2h_bh_{b-1}}\geq 0$ also holds. 
 
 Therefore, for constants $0<\epsilon_1\leq \frac12$ and $\epsilon_2=1$, we have $[\mathcal{A}_{d^*} + \mathcal{A}^z]({\mathbf 1})\geq \mathbf 0$.
 In particular, we have a larger $\ell$ compared to constraints from $\mathbf A_d$. 
 \subsection{The main result} \label{main_constraints}
 
We have shown that for two constants $0<\epsilon_1 \leq \frac12$ and $\epsilon_2=1$, under mesh constraints \eqref{meshconstraint-1}, 
the matrices $A_{d^*}$, $A^z$, $A^s$ constructed above satisfy
 $(A_{d^*}+A^z) \mathbf 1\geq \mathbf 0$ and  $A_a^+ \leq A^zA_{d^*}^{-1}A^s$. 

For any fixed $\epsilon_1>0$ and $\epsilon_2=1$, $A^z$ also has the same sparsity pattern as $A$. Thus if $\ell$ in \eqref{meshconstraint-1} 
is replaced by $\sup_{0<\epsilon_1\leq \frac12, \epsilon_2=1} \ell(\epsilon_1, \epsilon_2)=4,$ 
 Theorem \ref{newthm3-r} still applies to conclude that $A^{-1}\geq 0$.

\begin{theorem}
The $Q^2$ spectral element method \eqref{q2-scheme-nonuniform}  has a monotone matrix $\bar L_h$ thus satisfies discrete maximum principle under the following mesh constraints:
 \begin{equation}
 \label{meshconstraint-2}
     \begin{gathered}
    h_ah_{a-1} \geq \frac{7}{12}max\{h_b^2,h_{b-1}^2\}, \quad h_bh_{b-1} \geq \frac{7}{12}max\{h_a^2,h_{a-1}^2\}, \\ min\{h_a,h_{a-1}\} \geq \sqrt{\frac{1}{3}}max\{h_b,h_{b-1}\}, \quad min\{h_b,h_{b-1}\} \geq \sqrt{\frac{1}{3}}max\{h_a,h_{a-1},\}
    \end{gathered}
\end{equation}
where $h_a, h_{a-1}$ are mesh sizes for $x$-axis and  $h_b, h_{b-1}$ are mesh sizes for $y$-variable in four adjacent rectangular cells  as shown in Figure \ref{mesh}.
\end{theorem}
 
\begin{rmk}
The following global constraint is sufficient to ensure \eqref{meshconstraint-2}:
\begin{gather}
    \frac{25}{32} \leq  \frac{h_m}{h_n} \leq \frac{32}{25},
\end{gather}
where $h_m$ and $h_n$ are any two grid spacings in a non-uniform grid generated from a non-uniform rectangular mesh for $Q^2$ elements. 
\end{rmk}

\begin{rmk}
Though the mesh constraints above may not be sharp, similar constraints are necessary for monotonicity, as will be shown in numerical tests in the next section.
\end{rmk}
\begin{rmk}
For $Q^1$ finite element method solving $-\Delta u=f$ to satisfy discrete maximum principle on non-uniform  rectangular meshes \cite{christie1984maximum},  the mesh constraints are 
\begin{equation}
    h_ah_{a-1} \geq \frac{1}{2}max\{h_b^2,h_{b-1}^2\} \quad h_bh_{b-1} \geq \frac{1}{2}max\{h_a^2,h_{a-1}^2\}.
\end{equation}
\end{rmk}

\section{Numerical Tests}
\label{sec:5}
\label{sec-test} 
 
\subsection{Accuracy tests}
We show some accuracy tests of the $Q^2$ spectral element method for solving $-\Delta u=f$ on a square $(0,1)\times{}(0,1)$ with Dirichlet boundary conditions.
This scheme is fourth order accurate
in $\ell^2$-norm over quadrature points on uniform meshes \cite{li2019fourth}.
 \textcolor{black}{On a quasi-uniform mesh,  we test the error in $\ell^\infty$-norm to show that this is indeed a high order accurate scheme, which is at least third order accurate. We remark that $Q^2$ spectral element method as a finite difference scheme in $\ell^\infty$ norm is not fourth order accurate even on a uniform mesh, due to the singularity in Green's function in multiple dimensions, see numerical results in \cite{li2019fourth} and references therein. }

Quasi-uniform meshes were generated by setting each pair of consecutive finite element cells along the axis to have a fixed ratio $\frac{h_k}{h_{k-1}} = 1.01$.   The scheme is tested for the following very smooth solutions:
\begin{enumerate}
\item The Laplace equation $-\Delta u=0$ with Dirichlet boundary conditions and  $u(x,y) = log((x+1)^2 + (y+1)^2) + sin(y)e^x$.
\item Poisson equation $-\Delta u=f$ with homogeneous Dirichlet boundary condition:
\begin{equation}\label{numtestpoisson1}
\begin{gathered}
f(x,y) = 13\pi^2sin(3\pi{}y)sin(2\pi{}x) + 2y(1-y) + 2x(1-x) \\ u(x,y) = sin(3\pi{}y)sin(2\pi{}x) + xy(1-x)(1-y)
\end{gathered}
\end{equation}
\item Poisson equation $-\Delta u=f$ with nonhomogeneous Dirichlet boundary condition:
\begin{equation}\label{numtestpoisson3}
\begin{gathered}
f = 74\pi{}^2cos(5\pi{}x)cos(7\pi{}y) - 8 \\
u = cos(5\pi{}x)cos(7\pi{}y) + x^2 + y^2
\end{gathered}
\end{equation} 
\end{enumerate}

The  errors of $Q^2$ spectral element method on quasi uniform rectangular meshes are listed in  Table \ref{table-q2}.

%

\begin{table}[htbp]
\centering
\caption{Accuracy test on quasi-uniform meshes.}
\renewcommand{\arraystretch}{2}
\resizebox{0.7\textwidth}{!}{
\begin{tabular}{|c|c|c|c| }
\hline
\multirow{2}{*}{Finite Difference Grid} & 
\multirow{2}{*}{ Ratio $\frac{h_i}{h_{i-1}}$} & 
\multicolumn{2}{c|}{$Q^2$ spectral element method}  
 \\ 
\cline{3-4}
 & &  $l^\infty$ error & order   
   \\
\hline
\multicolumn{4}{|c|}{test on $-\Delta u=0$}  \\
\hline
$7\times 7$  & 
1.01 & 
 2.66E-5     &     -              \\
\hline
$15\times 15$ & 
1.01 & 
    1.97E-6     &     3.74               \\
\hline
$31\times 31$ &
1.01 & 
    1.54E-7     &     3.67               \\
\hline
$63\times 63$ &
1.01 & 
    1.37E-8     &     3.49                \\ \hline
\multicolumn{4}{|c|}{test on \eqref{numtestpoisson1}}  \\
\hline
$7\times 7$  & 
1.01 & 
        4.92E-2     &     -                   

\\
\hline
$15\times 15$ & 
1.01 & 
        3.19E-3     &     3.94              \\
\hline
$31\times 31$ &
1.01 & 
    2.29E-4     &     3.79             \\
\hline
$63\times 63$ &
1.01 & 
    1.80E-5     &     3.67              \\ \hline
\multicolumn{4}{|c|}{test on \eqref{numtestpoisson3}}  \\
\hline
$7\times 7$  & 
1.01 & 
        1.20E-0     &     -               \\
\hline
$15\times 15$ & 
1.01 & 
        1.03E-1     &     3.54            \\
\hline
$31\times 31$ &
1.01 & 
9.10E-3     &     3.50             \\
\hline
$63\times 63$ &
1.01 & 
 9.64E-4     &     3.23             \\ \hline
\end{tabular}}
\label{table-q2}
\end{table}

\subsection{Necessity of Mesh Constraints}
Even though the mesh constraints derived in the previous section are only sufficient conditions for monotonicity,
in practice a mesh constraint is still necessary for the inverse positivity to hold. 
Consider a non-uniform $Q^2$ mesh with $5\times{} 5$ cells on the domain $[0,1]\times{}[0,1]$, which has a  $9\times{}9$ grid for the interior of the domain. Let the mesh on both axes be the same and let the four outer-most cells for each dimension be identical with length $2h$. Then the middle cell has size $2h^\prime{}\times 2h^\prime$ with $h^\prime=\frac12-2h$. Let the ratio $h'/ h$ increase gradually from $h'/ h = 1$ (a uniform mesh) until the minimum value of the inverse of the matrix becomes negative. Increasing by values of $0.05$, we obtain the first negative entry of $\bar L_h^{-1}$ at $h'/h= 5.35$ with $h = 0.0535$ and $h^\prime{}=0.2861$,
and such a mesh is shown in Figure \ref{inverse_entries} (a). Figure \ref{inverse_entries} (b) shows how the smallest entry of $\bar L_h^{-1}$ decreases as $h'/ h$ increases.

   \begin{figure}[htbp]
 \subfigure[A non-uniform mesh with $5\times 5$ cells on which the $Q^2$ spectral element method is no longer monotone. The minimum value of $\bar L_h^{-1}$  is $-6.14E-8$.]{\includegraphics[scale=0.6]{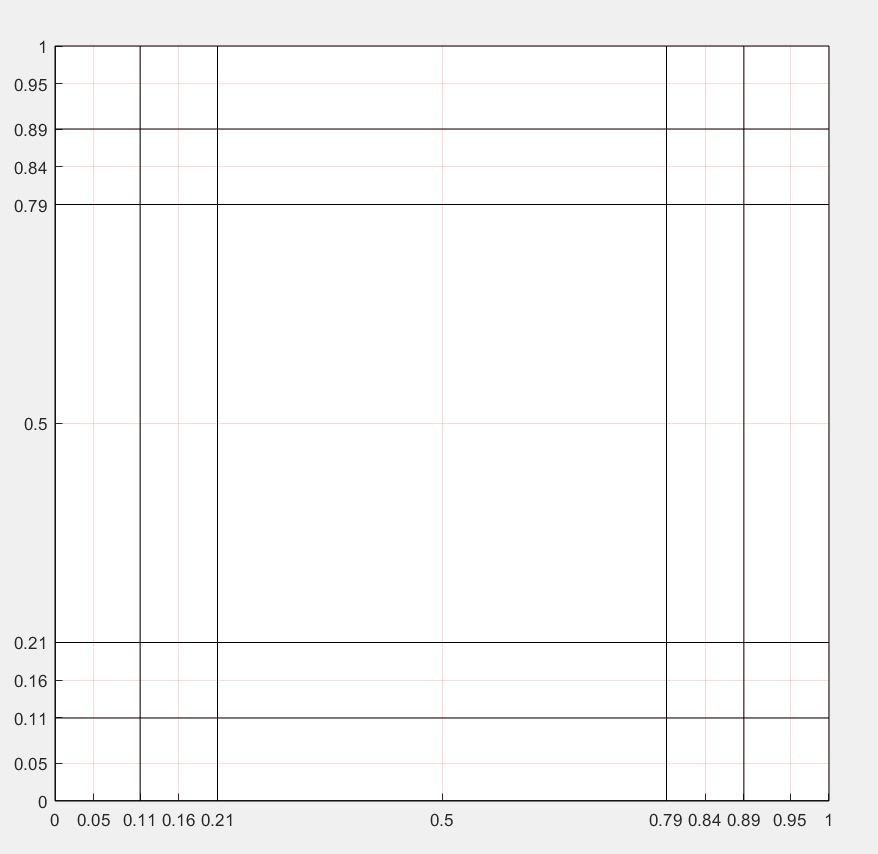} } \\
 \subfigure[A plot of the minimum value of $\bar L_h^{-1}$ as $h'/h$ increases.  ]{\includegraphics[scale=0.7]{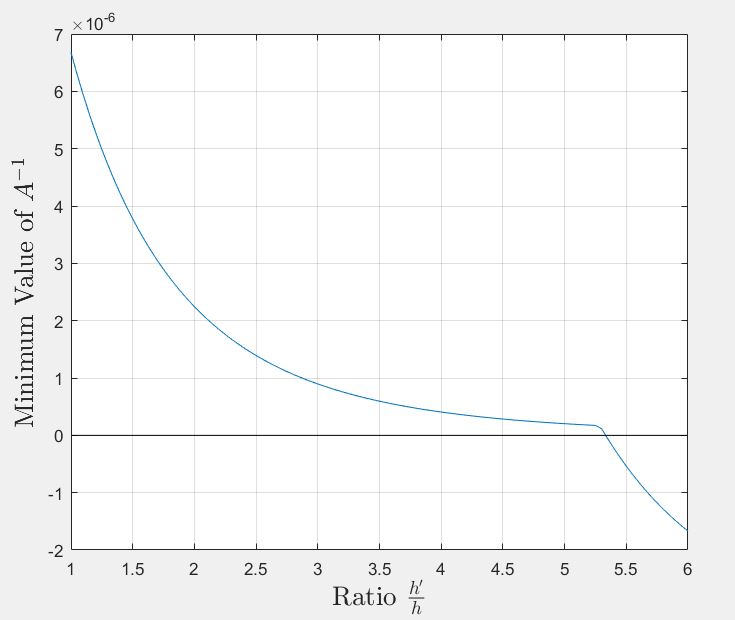}}
 \caption{Necessity of mesh constraints for inverse positivity $\bar L_h^{-1}\geq 0$ where $\bar L_h$ is the matrix in $Q^2$ spectral element method on non-uniform meshes. }
 \label{inverse_entries}
 \end{figure}

 \section{Concluding remarks}
 \label{sec:remark}
By verifying a relaxed  Lorenz's condition, we have
discussed suitable mesh constraints, under which the   $Q^2$ spectral element method on quasi-uniform meshes 
 is monotone.  Even though the derived mesh constraints may not be sharp, a similar constraint is necessary for the monotonicity to hold. 
\begin{appendix}
    \section{Appendix: M-matrices}
 Nonsingular M-matrices are inverse-positive matrices. There are many equivalent definitions or characterizations of M-matrices, see 
\cite{plemmons1977m}. 
The following is a convenient sufficient but not necessary characterization of nonsingular M-matrices \cite{li2019monotonicity}:
\begin{theorem}
\label{rowsumcondition-thm}
For a real square matrix $A$  with positive diagonal entries and non-positive off-diagonal entries, $A$ is a nonsingular M-matrix if  all the row sums of $A$ are non-negative and at least one row sum is positive. 
\end{theorem}
 By condition $K_{35}$ in \cite{plemmons1977m}, a sufficient and necessary characterization  is,
 \begin{theorem}
\label{rowsumcondition-thm2}
 For a real square matrix $A$  with positive diagonal entries and non-positive off-diagonal entries, $A$ is a nonsingular M-matrix if  and only if that there exists a positive diagonal matrix 
 $D$ such that $AD$ has all
positive row sums.
 \end{theorem}
\begin{rmk}
Non-negative row sum is not a necessary condition for M-matrices. For instance, the following matrix $A$ is an M-matrix by Theorem \ref{rowsumcondition-thm2}:
  $$A = \begin{bmatrix}
                       10  & 0 & 0\\
                       -10 & 2 & -10\\
                       0   & 0  & 10 \end{bmatrix} 
,D =
 \begin{bmatrix}
                     0.1 &   0 & 0\\
                       0 & 2   & 0\\
                       0 &   0 & 0.1 \end{bmatrix}
,AD =
 \begin{bmatrix}
                     1   &   0 & 0\\
                      -1 &  4  & -1\\
                       0 &   0 & 1 \end{bmatrix}.
$$
\end{rmk}
 
\end{appendix}

\section*{Acknowledgments}

\bibliographystyle{plain}
\bibliography{reference.bib}
\end{document}